\documentclass[12pt,a4paper]{amsart}
\usepackage{latexsym}
\usepackage[latin1]{inputenc}
\usepackage{graphicx}
\usepackage{amssymb,amsmath,amsfonts}
\usepackage{enumerate}
\usepackage{amsthm,upgreek}

\usepackage{srcltx}

\theoremstyle{definition}
\newtheorem{definition}{Definition}[section]

\theoremstyle{plain}
\newtheorem{theorem}[definition]{Theorem}
\newtheorem{proposition}[definition]{Proposition}
\newtheorem{corollary}[definition]{Corollary}
\newtheorem{lemma}[definition]{Lemma}
\newtheorem*{theorem*}{Theorem}

\theoremstyle{remark}

\newtheorem{remark}[definition]{Remark}

\newcommand{\D}{\mathcal{D}}
\newcommand{\C}{\mathbb{C}}

\newcommand{\N}{\mathbb{N}}
\newcommand{\Id}{\mathrm{Id}}
\newcommand{\re}{\mathrm{Re}}
\newcommand{\im}{\mathrm{Im}}
\newcommand{\vol}{\mathrm{vol}}
\renewcommand{\L}{\mathcal{L}}

\newcommand{\Mn}{M_\epsilon^n}

\newcommand{\func}[5]{\ensuremath{\begin{array}{cccl}
#1:&#2&\longrightarrow&#3\\&#4&\mapsto&#5\end{array}}}

\title[Crofton formulas in complex space forms]{The Gauss-Bonnet theorem and Crofton type formulas in complex space forms}
\author{Judit Abardia} 
\address{Institut f\"ur Mathematik, Johann Wolfgang Goethe-Universit\"at Frankfurt, 
Robert-Mayer-Str. 10, 60054 Frankfurt, Germany}
\email{abardia@math.uni-frankfurt.de}
\author{Eduardo Gallego} 
\author{Gil Solanes}
\address{Departament de Matem{\`a}tiques, Facultat de Ci{\`e}ncies,
Universitat Aut{\`o}noma de Barcelona, 08193--Bellaterra (Barcelona), Spain}
\email{egallego@mat.uab.cat, solanes@mat.uab.cat}

\thanks{Work partially supported by FEDER/MEC grant MTM2006-04353, the Schweizerischer Nationalfonds grants PP002-114715, the Ram\'on y Cajal program, the DIUR from the Generalitat de Catalunya and the ESF}
\subjclass{Primary 53C65; Secondary 52A22, 53C55}
\keywords{Complex space forms, Gauss-Bonnet formula, Crofton formulas, Valuations}

\begin{document}

\maketitle
\begin{abstract}We compute the measure with multiplicity of the set of complex planes intersecting a compact domain in a complex space form. The result is given in terms of the so-called hermitian intrinsic volumes. Moreover, we obtain two different versions for the Gauss-Bonnet-Chern formula in complex space forms. One of them gives the Gauss curvature integral in terms of the Euler characteristic, and some hermitian intrinsic volumes. The other one, which is shorter, involves the measure of complex hyperplanes meeting the domain. As a tool, we obtain variation formulas in integral geometry of complex space forms. 
\end{abstract}

\section{Introduction} 

In this paper we obtain some basic formulas in the integral geometry of complex space forms. In euclidean space, integral geometry was initiated by Blaschke and his school (cf. \cite{blaschke}). The extension to real space forms is mainly due to Santal\'o (cf. \cite{santalo.primera}). Recently, Alesker (cf. \cite{alesker.HardLefschetz}),  Bernig and Fu (cf. \cite{bernig.fu}) have made significant progress  in the integral geometry of standard hermitian space. Previously, Park (cf. \cite{park}) studied successfully the case of complex space forms in dimensions $2$ and $3$. Here we begin the generalization to higher dimensions.

Let us recall Crofton's formula in the real space form $N_\epsilon^n$ of sectional curvature $\epsilon$ which reads as follows  (cf. \cite[(17.58) and (17.59)]{santalo.primera})
\begin{equation}\label{CroftonReal}
 \int_{\mathcal L_r}\chi(\Omega\cap L_r)dL_r=\sum_{j=0}^{[r/2]}c_j\epsilon^j\mu_{n-r+2j}(\Omega)
\end{equation}
where $\L_{r}$ is the space of $r$-dimensional totally geodesic planes endowed with a measure $dL_{r}$ invariant under the group of isometries. Here $\Omega$ belongs to some suitable class of  subsets of $N_\epsilon^n$. For instance, the class of compact domains with smooth boundary. The funcionals $\mu_j$ are the so-called {\em intrinsic volumes} appearing in the Steiner formula for the volume of parallel sets (cf. \cite{santalo.primera,gao.hug.schneider}). When $\partial \Omega$ is smooth, $\mu_j(\Omega)$ is a multiple of the $(n-j-1)$-th mean curvature integral of $\partial \Omega$. The coefficients $c_{j}$ are known and depend only on $n$, $r$ and $j$. 

\smallskip
Our main result is a generalization of \eqref{CroftonReal} to the complex space forms $\Mn$, i.e. the  $n$-dimensional complete simply connected K\"ahler manifolds of constant holomorphic curvature $4\epsilon$. 
For $\epsilon>0$ (resp. $\epsilon<0$),  $\Mn$ is isometric to the complex projective (resp. complex hyperbolic) space endowed with the Fubini-Study metric (resp. Bergmann metric) suitably rescaled, while $M_0^n$ is the standard hermitian space $\C^n$. 

Let $\L^\C_r$ denote the space of totally geodesic complex submanifolds of $\Mn$ of complex dimension $r$ (complex $r$-planes). This is a homogeneous space under the isometry group of $\Mn$, and admits an invariant measure $dL_r$ which is unique up to normalization. Here we take the normalization coming from the Maurer-Cartan forms (see Lemma \ref{lema.densitat}).  Our main goal is to determine
\begin{equation}\label{defval}
 \phi_{n-r}(\Omega):=\int_{\L_r^\C}\chi(\Omega\cap L_r)dL_r,\qquad r=1,\cdots, n-1,
\end{equation}
in terms of the geometry of $\Omega\subset \Mn$. Again there are different choices for the class of subsets of $\Mn$ where $\Omega$ is taken. Here we concentrate mainly on $\mathcal P(\Mn)$, the class of compact submanifolds with corners (cf. \cite{alesker.bernig}). In particular this class contains the compact domains with smooth boundary. 

Our results are better described in the language of valuations. In vector spaces, valuations are defined as finitely additive functionals on convex sets. Hadwiger's characterization theorem of continuous invariant valuations allowed an axiomatic approach to integral geometry in euclidean space. Recently, a theory of valuations on manifolds has been developed by Alesker (cf. \cite{alesker1,alesker2,alesker3,alesker4}).  

Valuations on manifolds (also called smooth valuations) are defined  as functionals on the class of compact submanifolds with corners fulfilling some additivity and smoothness conditions.  For simplicity we use the following characterization. Let $X$ be an $n$-dimensional riemannian manifold. Given a submanifold with corners $\Omega\in \mathcal P(X)$, we consider the so-called unit normal cycle $N(\Omega)$ which is a Lipschitz submanifold of the unit tangent bundle $S(X)$. If $\partial \Omega$ is smooth, $N(\Omega)$ is just the outer unit normal bundle of $\partial \Omega$.  A valuation $\mu$ on $X$ is obtained  by integration of a smooth measure $\eta$ of $X$, and an $(n-1)$-form $\omega$ of $S(X)$ as follows 
 $$\mu(\Omega)=\int_{\Omega}\eta + \int_{N(\Omega)}\omega.$$ 
For instance, the intrinsic volumes $\mu_j$ appearing in \eqref{CroftonReal} are of this form. Also the Euler characteristic $\chi$ is a valuation in this sense, as shown by Chern's proof of the Gauss-Bonnet theorem \cite{chern.curvatura.integra}. Furthermore $\phi_{n-r}$ is a valuation. For $\epsilon>0$ this follows by the results in \cite{fu.indiana}. For $\epsilon=0$, it can be shown by the same results using a limit procedure. The case $\epsilon<0$ will follow from Theorem \ref{mesuresCKn} in this paper.

\smallskip
If the isometry group $G$ of $X$ acts transitively on $S(X)$, then the vector space of $G$-invariant valuations is finite dimensional (cf. \cite{fu.indiana}). In hermitian space $\C^n$, several bases of $U(n)$-invariant valuations have been obtained (cf. \cite{alesker.HardLefschetz,bernig.fu}). One of these bases is composed of the so-called {\em hermitian intrinsic volumes} $\{\mu_{k,q}\}_{k,q}$ (cf. section 2 for the definition). These valuations $\mu_{k,q}$ generalize in a natural way to complex space forms $\Mn$. As a consequence (at least for $\epsilon\geq 0$), the valuations $\phi_{j}$ and also $\chi$ are linear combinations of the hermitian intrinsic volumes $\mu_{k,q}$. Our main result gives these combinations explicitly.

\begin{theorem*}\label{teorema}In $\Mn$, the following equality between valuations holds for $r=1,\dots,n-1$
\begin{align}\label{Crofton}\notag\phi_{n-r}=\frac{\vol(G_{n-1,r}^{\C})}{\binom{n-1}{r}}&\left(\sum_{j=n-r}^{n}\epsilon^{j-(n-r)}\omega_{2n-2j}\binom{n}{j}^{-1}\!\!\!\cdot\huge((j+r-n+1)\mu_{2j,j}+\right.\\&\qquad+\left.\sum_{q=\max\{0,2j-n\}}^{j-1}\frac{1}{4^{j-q}}\binom{2j-2q}{j-q}\mu_{2j,q}\huge)\right),
\end{align} 
where $\omega_{i}$ denotes the $i$-dimensional volume of the euclidean unit ball, and $G_{n-1,r}^\C$ is the grassmannian of complex linear $r$-planes in $\C^{n-1}$. Moreover,
\begin{equation}\label{gbIntro}\chi=\sum_{c=0}^{n}\epsilon^c\frac{\omega_{2n-2c}}{\binom{n}{c}\omega_{2n}}\!\!\left(\sum_{q=\max\{0,2c-n\}}^{c-1}\!\!\frac{1}{4^{c-q}}\binom{2c-2q}{c-q}\mu_{2c,q}\!+\!(c+1)\mu_{2c,c}\!\right)\!\!.\end{equation}
\end{theorem*}

Using the results in \cite{bernig.fu} one can prove formula \eqref{Crofton} for $\epsilon=0$. Here we give an alternative proof, and we extend the result to $\epsilon\neq 0$.  In case $r=1$, equality (\ref{Crofton}) has already been proved by different methods in \cite{abardia.m1}. We would like to remark that formula \eqref{Crofton} answers a question about convex sets of $\mathbb C^n$ posed by Naveira  in \cite{naveira}.

Equality (\ref{gbIntro}) is a version of   the Gauss-Bonnet theorem in complex space forms. There seems to be no direct way to prove such an explicit formula from the intrinsic Gauss-Bonnet theorem of  \cite{chern.curvatura.integra}. In complex dimensions $n=2,3$, formula \eqref{gbIntro} was obtained in \cite{park}.

Combining expressions (\ref{Crofton}) and (\ref{gbIntro}) we obtain 
\begin{align*}\omega_{2n}\chi&=\epsilon\phi_1+\sum_{k=0}^{n}\epsilon^k \omega_{2n-2k}\binom{n}{k}^{-1}\mu_{2k,k}.\end{align*}

This expression is similar to the following one for real space forms $N_\epsilon^n$ (cf. \cite{solanes.gb}) 
$$\omega_{n-1}\chi=2n\epsilon{\omega_n}\,\varphi_2+{\omega_{n-1}}\mu_0,$$
where $\varphi_2$ is the valuation given by the left hand side of \eqref{CroftonReal} for $r=n-2$.

The main idea for the proof of the theorem above is to compare the first variations of both sides of equalities (\ref{Crofton}) and (\ref{gbIntro}). Indeed, valuations are determined by their first variation up to scalar multiples of $\chi$.

In order to obtain the variation of $\phi_{j}$ we proceed as in \cite{solanes.gb} (see Section \ref{variation1}). The variation of the hermitian intrinsic volumes $\mu_{k,q}$ was obtained by Bernig and Fu in \cite{bernig.fu} in case $\epsilon=0$. Here we use the same method to find this variation for $\epsilon\neq 0$ (see Section \ref{variation2}).

\section*{Acknowledgments}We wish to thank Andreas Bernig and Joseph Fu for illuminating discussions during the preparation of this work.

\section{Hermitian intrinsic volumes}\label{two}
Let $\Mn$ be a (simply connected) complex space form with constant holomorphic curvature $4\epsilon$. We denote by $\langle\ ,\ \rangle$ the riemannian metric, and by $S(\Mn)$ the unit tangent bundle of $\Mn$.

Recall that $\mathcal P(\Mn)$ denotes the class of compact submanifolds with corners of $\Mn$ (cf. \cite{alesker.bernig}). For technical reasons we also consider the following class.

\begin{definition}
 We denote by $\mathcal R(\Mn)$ the class of compact subsets $\Omega\subset\Mn$ such that the boundary $\partial \Omega$ is a $C^{1,1}$ regular hypersurface (i.e. it is locally the graph of a $C^1$ function with Lipschitz gradient), which is moreover $C^\infty$ on an open full-measure subset.
\end{definition}

In particular, the parallel sets (at small distances) of an element of $\mathcal P(\Mn)$ belong to $\mathcal R(\Mn)$. Both classes $\mathcal P(\Mn)$, and $\mathcal R(\Mn)$ contain the compact domains with smooth boundary, and are contained in the class of sets of positive reach. The latter is guaranteed by the $C^{1,1}$ condition on $\mathcal R(\Mn)$ (cf. \cite[Lemma 4.11]{federer}).  The smoothness condition is needed in subsection \ref{variation2}.

\begin{definition}Let $\Omega\subset \Mn$ belong to the  class  $\mathcal P(\Mn)$ or $\mathcal R(\Mn)$. The  {\em normal cycle} of $\Omega$ is defined as 
\begin{align*}
N(\Omega):=\{(p,v)&\in S(\Mn)\ |\ \langle v,x'(0)\rangle\leq 0,\\ &\forall x:[0,1)\rightarrow \Omega \mbox{ smooth with } x(0)=p \}.
\end{align*}
This is a Lipschitz submanifold of $S(\Mn)$. Hence it makes sense to integrate $(2n-1)$-differential forms on it. Thus, $N(\Omega)$ defines a current. This current is a cycle in the sense that vanishes on exact forms (cf. for instance the survey \cite[Prop. 29]{thaele}).

Given a $(2n-1)$-form  $\omega$ in $S(\Mn)$, and a smooth measure $\eta$ we may consider, for every  $\Omega\in \mathcal P(\Mn)$ or $\mathcal R(\Mn)$ $$\int_{\Omega}\eta + \int_{N(\Omega)}\omega.$$ The resulting functionals on $\mathcal{P}
(\Mn)$ and on $\mathcal R(\Mn)$ are called \emph{smooth valuations}.
\end{definition}

Let $(z,e_{1})\in S(\Mn)$ and let $\{z;e_{1},\dots,e_{n}\}$ be a moving frame defined on an open neighborhood  $U\subset S(\Mn)$. We denote by $\{\omega_{1},\omega_{2},\dots,\omega_{n}\}$ the 1-forms in $S(\Mn)$ defined as the dual basis of $\{e_{1},\dots,e_{n}\}$, and by $\{\omega_{ij}\}$ the corresponding connection forms of $\Mn$. That is, if $(\,,)$ denotes the Hermitian product on $\Mn$ and $\nabla$ the Levi-Civita connection, then \begin{align}\label{formes1}\omega_{j}=(dz,e_{j})\quad\text{and}\quad\omega_{jk}=(\nabla e_{j},e_{k})\quad\text{where}\quad j,k\in\{1,\dots,n\}.\end{align} Thus, these forms are $\C$-valued. We denote
\begin{align}\label{newLabel}\omega_{j}&=\alpha_{j}+i\beta_{j},\\
\omega_{jk}&=\alpha_{jk}+i\beta_{jk}.\nonumber\end{align}
Forms $\alpha_{1}$, $\beta_{1}$ and $\beta_{11}$ are global forms in $S(\Mn)$. We denote them by $\alpha$, $\beta$, $\gamma$ respectively.
Note that $\alpha$ coincides with the contact form of the unit tangent bundle $S(\Mn)$.

The following lemma is a well-known property of the normal cycle.
\begin{lemma}\label{anulen}Let $\Omega\subset \Mn$ belong to the  class  $\mathcal P(\Mn)$ or $\mathcal R(\Mn)$. Then $\alpha$ and $d\alpha$ vanish (almost everywhere) on $N(\Omega)\subset S(\Mn)$ (i.e. $N(\Omega)$ is Legendrian).
\end{lemma}
\begin{proof}Let $(p,v)\in N(\Omega)$ be an element with tangent space $T_{(p,v)}N(\Omega)$, and let $V\in T_{(p,v)}N(\Omega)$. Then, $\alpha(V)_{(p,v)}=\langle d\pi(V),v\rangle=0$ where $\pi:S(\Mn)\rightarrow \Mn$ is the canonical projection (cf. \cite{federer}). Hence, $\alpha$ vanishes on an open full-measure subset $U\subset N(\Omega)$. Trivially, $d\alpha$ vanishes on $U$.
\end{proof}

Consider the following invariant 2-forms in $S(\Mn)$ 
\begin{align}\label{defRefMob}\nonumber
\theta_{0}&=-\im((\nabla e_{1},\nabla e_{1}))=\sum_{i=2}^{n}\alpha_{1i}\wedge\beta_{1i}\\
\theta_{1}&=-\im((dz,\nabla e_{1})-(\nabla e_{1},dz))=\sum_{i=2}^{n}(\alpha_{i}\wedge\beta_{1i}-\beta_{i}\wedge\alpha_{1i})\\
\nonumber\theta_{2}&=-\im((dz,dz))=\sum_{i=2}^{n}\alpha_{i}\wedge\beta_{i}\\
\nonumber\theta_{s}&=\re((dz,\nabla e_{1})-(\nabla e_{1},dz))=\sum_{i=2}^{n}(\alpha_{i}\wedge\alpha_{1i}+\beta_{i}\wedge\beta_{1i}).
\end{align}

\begin{remark}These forms coincide with the invariant $2$-forms, $\theta_{0}$, $\theta_{1}$, $\theta_{2}$ and $\theta_{s}$ defined in $S(\C^n)$ by Bernig and Fu \cite{bernig.fu}. Note that, $\theta_{s}$ is the symplectic form of $T\Mn$. 
Park considered in \cite{park} similar 2-forms in $S(\Mn)$. Furthermore, it was shown there that $\alpha,\beta,\gamma,\theta_0,\theta_1,\theta_2,\theta_s$ generate the algebra of invariant forms of $S(\Mn)$.\end{remark}

Next we recall the exterior derivative of these forms, which can be found in \cite{bernig.fu} when $\epsilon=0$, or in \cite{park} for general $\epsilon$. 
\begin{lemma}[\cite{park}]\label{diferencials}In $S(\Mn)$ it is satisfied
$$\begin{array}{ll}
d\alpha=-\theta_{s}, & d\theta_{0}=-\epsilon(\alpha\wedge\theta_{1}+\beta\wedge\theta_{s}),\\
d\beta=\theta_{1}, & d\theta_{1}=0,\\
d\gamma=2\theta_{0}-2\epsilon\theta_{2}-2\epsilon\alpha\wedge\beta,\,\quad & d\theta_{2}=0.
\end{array}$$
\end{lemma}
\begin{proof}
This follows form the structure equations on $\Mn$, obtained by differentiating \eqref{formes1}.
\end{proof}

The following definition already appears in \cite{bernig.fu} when $\epsilon=0$.

\begin{definition}For positive integers $k,q\in\N$ with $\max\{0,k-n\}\leq q \leq\frac{k}{2}< n$, we define the following $(2n-1)$-forms in $S(\Mn)$
$$\beta_{k,q}:=c_{n,k,q}\beta\wedge\theta_{0}^{n-k+q}\wedge\theta_{1}^{k-2q-1}\wedge\theta_{2}^q\in \Upomega^{2n-1}(S(\Mn)),\quad \text{ if }k\neq 2q$$
$$\gamma_{k,q}:=\dfrac{c_{n,k,q}}{2}\gamma\wedge\theta_{0}^{n-k+q-1}\wedge\theta_{1}^{k-2q}\wedge\theta_{2}^q\in \Upomega^{2n-1}(S(\Mn)),\quad \text{ if }n\neq k-q$$
where
$$c_{n,k,q}:=\frac{1}{q!(n-k+q)!(k-2q)!\omega_{2n-k}}$$ and $\omega_{2n-k}$ denotes the volume of the $(2n-k)$-dimensional euclidean ball.
\end{definition}
These forms consitute a basis of the vector space of  isometry invariant $(2n-1)$-forms of $S(\Mn)$ modulo the subspace of multiples of $\alpha,d\alpha$.

For $\max\{0,k-n\}\leq q\leq \frac k 2 <n$ let us consider the valuations $\mu_{k,q}^\beta,\mu_{k,q}^\gamma$ given by
$$\mu_{k,q}^\beta(\Omega):=\int_{N(\Omega)}\beta_{k,q}\quad  (\mbox{if }k\neq 2q),$$$$\quad\mu^\gamma_{k,q}(\Omega):=\int_{N(\Omega)}\gamma_{k,q}\quad (\mbox{if }n\neq k-q).$$

In $\C^n$, it turns out that $\mu^\beta_{k,q}=\mu^\gamma_{k,q}$ (cf. \cite{bernig.fu}). 
Next, we give an analogous relation among $\{\mu^{\beta}_{k,q}\}$ and $\{\mu^{\gamma}_{k,q}\}$ in $\Mn$.

\begin{proposition}\label{relacioGammaB}In $\Mn$, for any positive integers $k,q$ such that $\max\{0,k-n\}< q< k/2< n$ it is satisfied $$\mu^{\gamma}_{k,q}=\mu^{\beta}_{k,q}-\epsilon\frac{ c_{n,k,q}}{c_{n,k+2,q+1}}\mu^{\beta}_{k+2,q+1}.$$
\end{proposition}
\begin{proof}We denote by $I$ the ideal generated by $\alpha$, $d\alpha$ and the exact forms. The elements of $I$ vanish on normal cycles, since these are Legendrian.
Thus, it is enough to prove \begin{equation}\label{modul}\gamma_{k,q}\equiv \beta_{k,q}-\epsilon\frac{c_{n,k,q}}{c_{n,k+2,q+1}}\beta_{k+2,q+1}\quad  \mathrm{mod}\ I.\end{equation}

Consider the form $\eta=\beta\wedge\gamma\wedge\theta_{0}^{n-k+q-1}\theta_{1}^{k-2q-1}\theta_{2}^q$. 
By Lemma \ref{diferencials} it follows that modulo $I$
$$d\eta\equiv \gamma\theta_{0}^{n-k+q-1}\theta_{1}^{k-2q}\theta_{2}^q-2\beta\theta_{0}^{n-k+q}\theta_{1}^{k-2q-1}\theta_{2}^q+2\epsilon\beta\theta_{0}^{n-k+q-1}\theta_{1}^{k-2q-1}\theta_{2}^{q+1}.$$
Using the definition of $\gamma_{k,q}$ and $\beta_{k,q}$ we obtain the relation in (\ref{modul}).
\end{proof}

\begin{remark}
 In complex dimensions $n=2,3$, the previous relations were found in \cite{park}.
\end{remark}

\begin{definition}We define  (for $\max\{0,k-n\}\leq q\leq \frac k 2 <n$) \begin{equation}\label{mus}\mu_{k,q}:=\left\{\begin{array}{rl}\mu^{\beta}_{k,q}&\text{ if }k\neq 2q \\ \mu^{\gamma}_{2q,q}&\text{ if }k=2q.\end{array}\right.\end{equation}
These valuations will be called {\em hermitian intrinsic volumes}.
In order to simplify some expressions, we define $\mu_{2n,n}:=\vol$.\end{definition}

\begin{remark}There is some arbitrariness in the previous choice.  It could be that a different convention may simplify  some formulas. 
\end{remark}

\begin{proposition}\label{baseMn}The hermitian intrinsic volumes
 $\mu_{k,q}$ (with $\max\{0,k-n\}\leq q\leq \frac k 2 \leq n$) form a basis of the space of invariant valuations of $\Mn$.
\end{proposition}
\begin{proof}
 For $\epsilon=0$ this was proved in \cite{bernig.fu}. For general $\epsilon$, the forms $\beta_{k,q},\gamma_{k,q}$ span the vector space of  invariant $(2n-1)$-forms of $S(\Mn)$ quotiented by the subspace of multiples of $\alpha,d\alpha$. Hence  $\{\mu_{k,q}\}$ span the space of invariant valuations of $\Mn$. It remains only to see that the dimension of this space does not depend on $\epsilon$. This follows from the following observation of J.~Fu: the space of invariant valuations of an isotropic homogeneous riemannian manifold is (non-canonically) isomorphic to the space of invariant valuations of the tangent space at any point. This is a consequence of Corollary 3.1.7 in \cite{alesker2}.
\end{proof}

\section{Variation formulas}\label{three}
\subsection{Variation of hermitian intrinsic volumes}\label{variation1}
In order to study the variation of the hermitian intrinsic volumes on $\Mn$, we follow the method used by Bernig and Fu \cite{bernig.fu}. First, we recall the definition of the Rumin operator, introduced in \cite{rumin}, and the notion of   Reeb vector field in a contact manifold. 

\begin{definition}Given $\omega\in\Upomega^{2n-1}(S(\Mn))$, let $\alpha\wedge\xi\!\in\!\Upomega^{2n-1}(S(\Mn))$ be the unique form such that $d(\omega+\alpha\wedge\xi)$ is a multiple of $\alpha$ (cf. \cite{rumin}). Then the Rumin operator $D$ is defined as $$D\omega:=d(\omega+\alpha\wedge\xi).$$
\end{definition}

\begin{definition}Let $M$ be a contact manifold and let $\alpha$ be the contact form. The Reeb vector field $T$ is the unique vector field over $M$ such that
\begin{equation}\label{campReeb}\left\{\begin{array}{l}
i_{T}\alpha=1,
\\ \mathcal{L}_{T}\alpha=0.
\end{array}\right.\end{equation}\end{definition}

If the contact manifold is the unit tangent bundle of a riemannian manifold, then the Reeb vector field corresponds to the geodesic flow (cf. \cite[p. 17]{blair}). 
\begin{lemma}\label{contraccio}In $S(\Mn)$, it is satisfied
$$\begin{array}{ll}
i_{T}\alpha=1, & i_{T}\theta_{1}=\gamma,\\
i_{T}\theta_{2}=\beta,\quad & i_{T}\beta=i_{T}\gamma=i_{T}\theta_{0}=i_{T}\theta_{s}=0.
\end{array}$$
\end{lemma}
\begin{proof}The first equality comes directly from  (\ref{campReeb}). Moreover, $ i_T\theta_s=-i_Td\alpha=di_T\alpha-\mathcal L_T\alpha=0.$ As $T$ is the geodesic flow, we have $\alpha_{i}(T)=\beta_{i}(T)=0$ and $\alpha_{1i}(T)=\beta_{1i}(T)=0$, $i\in\{2,\dots,n\}$. By definition in (\ref{defRefMob}), we obtain the result. 
\end{proof}

Given a smooth valuation $\mu$, and a vector field $X$ with flow $F_t$, we are interested in computing
\[
 \delta_X\mu(\Omega):=\left.\frac{d}{dt}\right|_{t=0}\mu(F_t(\Omega)).
\]
This can be done by means of the following result of \cite{bernig.fu}.

\begin{lemma}[Corollary 2.6 \cite{bernig.fu}]\label{variacio} Let $\Omega\subset \Mn$ belong to the class $\mathcal R(\Mn)$. Let $X$ be a smooth vector field on $\Mn$, and let $\mu$ be a smooth valuation given by a $(2n-1)$-form $\omega$ in $S(\Mn)$. Then $$\delta_X \mu (\Omega) = \int_{\partial\Omega} \langle X,\mathbf{n}\rangle\, \varphi^{*} (i_{T}(D\omega))$$ where $T$ is the Reeb vector field on $S(\Mn)$, $D\omega$ is the Rumin operator of $\omega$, $\mathbf{n}$ is the outward unit normal field, and $\varphi:\partial\Omega\rightarrow N(\Omega)$ is given by $\varphi(p)=(p,\mathbf{n}(p))$.
\end{lemma}

Although this result was stated in \cite{bernig.fu} for vector spaces, the proof given there is valid in an arbitrary riemannian manifold. 

By the previous lemma, the first variation of valuations is essentialy given by the following operator.

\begin{definition}We define a linear map $\delta$ from the space of smooth valuations to the quotient space $\Upomega^{2n-1}(\Mn)/ (\alpha,d\alpha)$ by$$\delta\mu=[i_{T}D\omega]$$ for $\mu$ defined by a $(2n-1)$-form $\omega\in\Upomega^{2n-1}(\Mn)$, and
$$\delta\mu=[2\lambda  \beta_{2n-1,n-1}]$$
for $\mu$ given by a smooth measure $\lambda \mathrm{vol}$. \end{definition}
This way, for every smooth valuation $\mu$ 
\begin{equation}\label{deltax}
 \delta_X\mu(\Omega)=\int_{\partial \Omega}\langle X,\mathbf{n}\rangle \varphi^*(\delta\mu).
\end{equation}

Further, we define $$B_{k,q}=[\beta_{k,q}], \Gamma_{k,q}=[\gamma_{k,q}]\in\Upomega^{2n-1}(\Mn)/ (\alpha,d\alpha).$$

\medskip
From Lemma \ref{diferencials}, we obtain the exterior differential of the forms $\beta_{k,q}$ and $\gamma_{k,q}$.

\begin{lemma}\label{derivadesE}In $\Mn$
$$d\beta_{k,q}=c_{n,k,q}(\theta_{0}^{n-k+q}\wedge\theta_{1}^{k-2q}\wedge\theta_{2}^q-\epsilon(n-k+q)\alpha\wedge\beta\wedge\theta_{0}^{n-k+q-1}\wedge\theta_{1}^{k-2q}\wedge\theta_{2}^q)$$
and
\begin{align*}d\gamma_{k,q} =& c_{n,k,q}(\theta_{0}^{n-k+q}\wedge\theta_{1}^{k-2q}\wedge\theta_{2}^q-\epsilon\theta_{0}^{n-k+q-1}\wedge\theta_{1}^{k-2q}\wedge\theta_{2}^{q+1}\\&-\epsilon\alpha\wedge\beta\wedge\theta_{0}^{n-k+q-1}\wedge\theta_{1}^{k-2q}\wedge\theta_{2}^q\\ &-\epsilon\frac{(n-k+q-1)}{2}\alpha\wedge\gamma\wedge\theta_{0}^{n-k+q-2}\wedge\theta_{1}^{k-2q+1}\wedge\theta_{2}^q\\ &-\epsilon\frac{(n-k+q-1)}{2}\beta\wedge\gamma\wedge\theta_{s}\wedge\theta_{0}^{n-k+q-2}\wedge\theta_{1}^{k-2q}\wedge\theta_{2}^q).\end{align*}
\end{lemma}

The variation of the valuations $\{\mu_{k,q}\}$ on $\C^n$ was found in \cite[Proposition 4.6]{bernig.fu}. We extend that result to $\Mn$ as follows. 

\begin{proposition}\label{variacions}In $\Mn$, for $k\neq 2q$
\small{\begin{align*}
&\delta \mu_{k,q}\! = 2c_{n,k,q}(c_{n,k-1,q}^{-1} (k-2q)^2 \Gamma_{k-1,q}\!-\!  
c_{n,k-1,q-1}^{-1} (n+q-k)q \Gamma_{k-1,q-1} \\
&+c_{n,k-1 ,q-1}^{-1}(n+q-k+\frac 1 2)q  B_{k-1,q-1}-
 c_{n,k-1,q}^{-1} (k-2q)(k-2q-1)  B_{k-1,q} \\
 &+\!\epsilon(c_{n,k+1,q+1}^{-1}(k-2q)(k-2q-1) B_{k+1,q+1}\!-\!c_{n,k+1,q}^{-1}(n-k+q)(q+\frac{1}{2}) B_{k+1,q}))
\end{align*}}
and
{\begin{align*}
 \delta\mu_{2q,q} = &2c_{n,2q,q}\Big(- 
c_{n,2q-1,q-1}^{-1} (n-q)q \Gamma_{2q-1,q-1}
\\&  
+ c_{n,2q-1 ,q-1}^{-1}(n-q+\frac 1 2)q  B_{2q-1,q-1}
\\
&+\epsilon c_{n,2q+1,q}^{-1}(n-q-1)(q+1)\Gamma_{2q+1,q}\\ &-\epsilon c_{n,2q+1,q}^{-1}((n-q)(2q+\frac{3}{2})-\frac{1}{2}(q+1)) B_{2q+1,q} \\
&+\epsilon^2 c_{n,2q+3,q+1}^{-1}(n-q-1)(q+\frac{3}{2}) B_{2q+3,q+1}
\Big).
\end{align*}}
\end{proposition}

\begin{proof}
We will use the following fact from the proof of Proposition 4.6 in \cite{bernig.fu}: for $\max\{0,k-n\}\leq q\leq k/2<n$ there exists an invariant form $\xi_{k,q}\in \Upomega^{2n-1}(S(\C^n))$ such that
\begin{equation}\label{dxi}
d\alpha\wedge \xi_{k,q}\equiv -\theta_0^{n-k+q}\theta_1^{k-2q}\theta_2^q\qquad \mathrm{mod} (\alpha),
\end{equation}
(we omit the wedges for readibility) and
\begin{align}\label{xi}
\xi_{k,q}  \equiv\,\, & \beta\gamma \theta_0^{n+q-k-1} \theta_1^{k-2q-2} \theta_2^{q-1}\\&\nonumber \wedge\left((n+q-k)q  \theta_1^{2}-(k-2q)(k-2q-1)\theta_0\theta_2 \right)\qquad \mathrm{mod} (\alpha,d\alpha).
\end{align}
In order to find $\delta\mu_{k,q}^{\beta}$ for general $\epsilon$, we take a form $\xi^\epsilon\in\Upomega^{2n-1}(S(\Mn))$ such that $\xi_{(p,v)}^\epsilon\equiv\xi_{k,q}(p',v')$ when we identify $T_{(p,v)}S(\Mn)$ and $T_{(p',v')}\C^n$, for every $(p,v)\in S(\Mn), (p',v')\in S(\C^n)$. Then, it is clear from Lemma \ref{derivadesE} and (\ref{dxi}) that $d(\beta_{k,q}+c_{n,k,q}\alpha\wedge\xi^\epsilon)\equiv 0$ modulo $\alpha$.

By Lemma \ref{diferencials}, the exterior differential of $\xi^\epsilon$ is 
\begin{align*}d\xi^\epsilon\equiv\,\, & \theta_{0}^{n+q-k-1}\theta_{1}^{k-2q-2}\theta_{2}^{q-1}((n-k+q)q\theta_{1}^2-(k-2q)(k-2q-1)\theta_{0}\theta_{2})\\&\wedge(\gamma\theta_{1}-2\beta\theta_{0}+2\epsilon\beta\theta_{2})\quad  \mathrm{mod}(\alpha,d\alpha)\end{align*}
and the contraction of $d\beta_{k,q}$ with respect to the field $T$, by Lemma \ref{contraccio}, is 
\begin{align*}i_{T}d\beta_{k,q}\equiv &\,\, c_{n,k,q}\theta_{0}^{n+q-k-1}\theta_{1}^{k-2q-1}\theta_{2}^{q-1}\\&\wedge((k-2q)\gamma\theta_{0}\theta_{2}+q\beta\theta_{0}\theta_{1}-\epsilon(n-k+q)\beta\theta_{1}\theta_{2})\quad  \mathrm{mod}(\alpha).\end{align*}

By substituting the last expressions in $i_{T}D\beta_{k,q}\equiv i_{T}d\beta_{k,q}-c_{n,k,q}d\xi^{\epsilon}$ ($\mathrm{mod}$ $\alpha$, $d\alpha$), we get the result. 

\vspace{0.3cm}

To compute $\delta\mu^\gamma_{2q,q}$, note that $d\gamma_{2q,q}$ has 3 terms which are not multiple of $\alpha$ (cf. Lemma \ref{derivadesE}). As before we consider $\xi_1^\epsilon,\xi_2^\epsilon\in \Upomega^{2n-1}(S(\Mn))$ corresponding to $\xi_{2q,q}$, and $\xi_{2q+2,q+1}$ respectively. Let us consider also
\begin{align}\label{tres}\xi_{3}^\epsilon=\,\, &  \frac{n-q-1}{2}\beta\gamma \theta_0^{n-q-2} \theta_2^{q}.\end{align}
Then the Rumin differential of $\gamma_{2q,q}$ is given by $D\gamma_{2q,q}=d(\gamma_{2q,q}+c_{n,2q,q}\alpha\wedge(\xi_1^\epsilon-\epsilon\xi_{2}^\epsilon-\epsilon\xi_{3}^\epsilon))$. Indeed, $d\alpha\wedge\xi_1^\epsilon$ cancels the first term of $d\gamma_{2q,q}$ modulo $\alpha$, and $d\alpha\wedge\xi_{2}^\epsilon$ cancels the second one. The third term is canceled exactly by $d\alpha\wedge\xi_{3}^\epsilon$. 

Now, using Lemmas \ref{derivadesE} and \ref{contraccio}
\[
i_Td\gamma_{2q,q}\!\equiv\! q \beta\theta_0^{n-q}\theta_2^{q-1}-\epsilon(q+2)\beta\theta_0^{n-q-1}\theta_2^q-\epsilon\frac{n-q-1}{2}\gamma\theta_0^{n-q-2}\theta_1\theta_2^q\quad \mathrm{mod} (\alpha,d\alpha).
\]
From (\ref{xi}) and (\ref{tres})
\[
d\xi_1^\epsilon\equiv (n-q)q\theta_0^{n-q-1}\theta_2^{q-1}(\gamma\theta_1-2\beta\theta_0+2\epsilon\beta\theta_2)\qquad \mathrm{mod} (\alpha,d\alpha).
\]
\[
d\xi_2^\epsilon\equiv (n-q-1)(q+1)\theta_0^{n-q-2}\theta_2^{q}(\gamma\theta_1-2\beta\theta_0+2\epsilon\beta\theta_2)\qquad \mathrm{mod} (\alpha,d\alpha).
\]
\[
d\xi_3^\epsilon\equiv \frac{n-q-1}{2}\theta_0^{n-q-2}\theta_2^{q}(\gamma\theta_1-2\beta\theta_0+2\epsilon\beta\theta_2)\qquad \mathrm{mod} (\alpha,d\alpha).
\]
Plugging this into $i_TD\gamma_{2q,q}\equiv i_Td\gamma_{2q,q}-c_{n,2q,q}(d\xi_1^\epsilon-\epsilon d\xi_2^\epsilon-\epsilon d\xi_3^\epsilon)$ mod $(\alpha,d\alpha)$ gives the result.
\end{proof}

\subsection{Variation of $\phi_{n-r}$}\label{variation2}
Recall that our goal is to determine 
\[
  \phi_{n-r}(\Omega)=\int_{\L_r^\C}\chi(\Omega\cap L_r)dL_r,\qquad r=1,\cdots, n-1
\]
where  $\L_{r}^{\C}$ is the space of totally geodesic complex submanifolds of complex dimension $r$ of $\Mn$. This is a homogeneous space with an invariant density given by the following lemma (as usual, $J$ denotes the complex structure).

\begin{lemma}[\cite{santalo.hermitian}]\label{lema.densitat}$\L_{r}^{\C}$ is a homogeneous space and $$\L_{r}^{\C}\cong U_{\epsilon}(n)/U_{\epsilon}(r)\times U(n-r)$$ where $$U_{\epsilon}(n)=\left\{\begin{array}{ll}\C^n\ltimes U(n), & \text{ if }\epsilon=0, \\ U(1+n), & \text{ if }\epsilon>0, \\ U(1,n), & \text{ if }\epsilon<0.\end{array}\right.$$
Let $\{g;g_{1},Jg_{1},\dots,g_{n},Jg_{n}\}$ be a local orthonormal frame associated to the elements of an open set $V\subset\L_{r}^{\C}$ such that $\{g_{n-r+1},Jg_{n-r+1},\dots,g_{n},Jg_{n}\}$ generate $T_{g}L$ for each $L\in V$. The invariant density of $\L_{r}^{\C}$ is given by 
\begin{equation}\label{densitat}
dL_{r}=\left|\bigwedge_{i=1}^{n-r}\alpha_{i}\wedge\beta_{i}
\mathop{\bigwedge_{i=n-r+1,\dots,n}}_{j=1,\dots,n-r}\alpha_{ij}\wedge\beta_{ij}\right|
\end{equation} where
$\{\alpha_i,\beta_i,\alpha_{i,j},\beta_{i,j}\}_{\{i,j\}}$ are defined as in \eqref{newLabel}.
\end{lemma}

The grassmannian $G_{n,r}^{\C}$ of complex $r$-planes in $\C^n$ can be identified with $\mathcal L_{r-1}^\C$ in $M_1^{n-1}\equiv\mathbb{CP}^{n-1}$.
With the normalization considered, the volume of $G_{n,r}^{\C}$ equals
$$\vol(G_{n,r}^{\C})={\vol(U(n)) \over \vol(U(r))\vol(U(n-r))}=\frac{\pi^{r(n-r)}1!2!\cdots (r-1)!}{(n-1)!(n-2)!\cdots (n-r)!}.
$$

\begin{remark}
 A different normalization for $dL_r$ was taken in \cite{bernig.fu}.  Indeed, the measure used there is 
\[
 dE_r=\frac{\binom{n-1}{r}}{\omega_{2n}\vol(G_{n-1,r}^{\C})}dL_r.
\]

\end{remark}

\bigskip On every $C^1$ real hypersurface $S\subset\Mn$ oriented by a unit normal $\mathbf{n}$ there is a canonical vector field given by $J\mathbf{n}$. There is also a distribution $\mathcal D=\mathrm{span} \{\mathbf{n},J\mathbf{n}\}^\bot$, so that $\mathcal D_x$ is the maximal complex linear subspace of $T_xM_\epsilon^n$ contained in $T_xS$ for every $x\in S$. We shall consider the bundle $G_{n-1,r}^{\C}(\D)$ whose fiber at every point $x\in S$ is the Grassmanian $G_{n-1,r}^{\C}(\D_{x})$ of $r$-dimensional complex subspaces of $\mathcal D_x$.

\begin{proposition}\label{variacioMesura}Let $\Omega\in\mathcal R(\Mn)$, and let $X$ be a smooth vector field on $\Mn$. Then $$\delta_{X}\phi_{n-r}(\Omega)
=\int_{\partial \Omega}\langle X, \mathbf{n}\rangle\left(\int_{G_{n-1,r}^{\C}(\D_{x})}\det(\mbox{\em II}|_{V})dV\right)dx$$ where $\mathbf{n}$ is the unit outward normal field and $\det(\mbox{\em II}|_{V})$ denotes the determinant of the second fundamental form $\mbox{\em II}$ of $\partial\Omega$ restricted to $V\in G_{n-1,r}^{\C}(\D_{x})$ being $\D$ the distribution of tangent complex hyperplanes.
\end{proposition}
\begin{proof}
We follow the same procedure as in \cite[Theorem 4]{solanes.gb}.

For every $V\in G_{n-1,r}^{\C}(\D_x)$, we make the parallel translation $V_t$ of $V$ along $F_t(x)$, the flow associated to $X$. Recall that parallel translation preserves the complex structure (cf. \cite[p. 326]{o'neill}). Then we project orthogonally $V_t$ onto $\mathcal D_{F_t(x)}$, obtaining a complex $r$-plane $V_t'$ (at least for small values of $t$).
We define $$\func{\psi}{G_{n-1,r}^{\C}(\D)\times(-\epsilon,\epsilon)}{\L_{r}^{\C}}{((x,V),t)}{\exp_{F_{t}(x)}V_t'.}$$ 

From Proposition 3 in \cite{solanes.gb} (whose proof works without change in our setting), we have
\[
\phi_{n-r}(\Omega_{h})-\phi_{n-r}(\Omega_0)=\int_{\mathcal L_r^\C}\sum \text{sign}\langle X,\mathbf{n}\rangle\,\text{sign} (\sigma_{2r}(\mbox{II}|_V))dL_r
\]
where the sum runs over the tangencies of a generic $L_r$ with the hypersurfaces $\partial \Omega_t$ with $0<t<h$. As $$\psi^*(dL_{r})=\iota_{\partial t}(\psi^*(dL_{r}))dt=\psi_{t}^*(\iota_{d\psi\partial t}dL_{r})dt$$ where $\psi_{t}=\psi(\cdot,t)$,  using the co-area formula (cf. \cite{federer})  we get 
\begin{equation}
 \label{formulot}\phi_{n-r}(\Omega_{h})-\phi_{n-r}(\Omega_0)=\int_0^h\int_{G^{\C}_{n-1,r}(\D)}\langle X,\mathbf{n}\rangle\,\text{sign} (\sigma_{2r}(\mbox{II}|_V))\psi_{t}^*(\iota_{dF\partial t}dL_{r})dt.
\end{equation}

Let $\{g;g_{1},Jg_{1},\dots,g_{n},Jg_{n}\}$ be a local orthonormal frame defined on $G_{n-1,r}^{\C}(\D)\times(-\epsilon,\epsilon)$ such that $g((x,l),t)=F(x,t)$, $g_{1}((x,l),t)$ is orthogonal to $\partial\Omega_{t}$ (at $F_t(x)$) and $\psi=\langle g,g_{n-r+1},Jg_{n-r+1},\dots,g_{n},Jg_{n}\rangle\cap\Mn$. We may assume the frame is defined in a neigborhhod of $\mathcal L_r$, since we are only interested in regular points of $\psi$.

Consider the curve $L_{r}(t)$ given by the parallel translation of $L_{r}$ along the geodesic given by $\mathbf{n}$, the outward normal vector to $\partial\Omega_{0}$. 
If $P\in T_{L_r}\mathcal L_r^\C$ denotes the tangent vector to $L_{r}(t)$ at  $t=0$, then 
\begin{align*}\omega_{i}(P)&=( dg(P),g_{i})=\Big(\left.\frac{d}{dt}\right|_{t=0}g(L_{r}(t)),g_{i}\Big)=0, 
\\\alpha_{1}(P)&=\langle dg(P),\mathbf{n}\rangle =1,
\\\omega_{kj}(P)&=( \nabla g_{k}(P),g_{j})=\left(\left.\nabla_{\partial \over\partial t} g_{k}(L_{r}(t))\right|_{t=0},g_{j}\right)=0\end{align*}
where $i\in\{2,\dots,n-r\}$, $j\in\{1,\dots,n-r\}$, and $k\in\{n-r+1,\dots,n\}$.
By (\ref{densitat}) and last equations we get the following equality between densities
$$dL_{r}=|\alpha_{1}|i_{P}dL_{r}$$
since $i_{P}dL_{r}=|\beta_1\bigwedge_{h=2}^{n-r}\alpha_{h}\wedge\beta_{h}\bigwedge\alpha_{ij}\wedge\beta_{ij}|$.
Thus, 
$$i_{\partial \psi\over\partial t}dL_{r}=\big|\alpha_{1}\Big({\partial \psi\over\partial t}\Big)\big|\, i_{P}dL_{r}+|\alpha_{1}|i_{\partial\psi\over\partial t}i_{P}dL_{r}$$ and
\begin{align*}\alpha_{1}\Big({\partial\psi\over\partial t}\Big)&=\big\langle dg\Big({\partial\psi\over\partial t}\Big),\mathbf{n}\big\rangle=\big\langle\frac{\partial F}{\partial t},\mathbf{n}\big\rangle,
\\\psi_{0}^*(\alpha_{1})(v)&=\langle dg(d\psi_{0}(v)),\mathbf{n}\rangle=0\quad\forall v\in T_{(p,V)}G_{n,r}^{\C}(T\partial\Omega_{0}).
\end{align*}
So, $$\psi_{0}^*(i_{{\partial\psi\over\partial t}}dL_{r})=|\langle \frac{\partial F}{\partial t},\mathbf{n}\rangle|\psi_{0}^*(i_{P}dL_{r}).$$

Finally,  using that $\psi_{0}^*(i_{P}dL_{r})=|\det(\mbox{II}|_V)|dVdx$, we get the result.
\end{proof}

\begin{remark}The integral
\[ \int_{G_{n-1,r}^\mathbb{C}} \det(\mbox{II}|_V)dV                                                                                           \]
seems difficult to compute  by direct means. However, we will find it by an indirect method (cf.~Remark \ref{solucio}). The analogous integral in real space forms is a multiple of an elementary symmetric function of the principal curvatures.\end{remark}

\section{Crofton type formulas}\label{four}
\subsection{In the standard Hermitian space}
Now we are ready to prove formula (\ref{Crofton}) for $\epsilon=0$. The following lemma will be used. 
\begin{lemma}\label{entries}Let $\Omega\subset\C^n$ be a compact domain with smooth boundary, and let $\varphi:\partial\Omega\rightarrow N(\Omega)$ the canonical map. Fix a point $x$ in $\partial\Omega$ and a reference $\{e_{1}=\varphi(x),e_{\overline{1}}=Je_{1},\dots,e_{n},e_{\overline{n}}=Je_{n}\}$ at $x$. Then $\varphi^{*}(\gamma_{k,q})=P_{k,q}dx$ where $dx$ is the volume element of $\partial\Omega$ and $P_{k,q}$ is a polynomial of degree $2n-k-1$ in the entries of the second fundamental form $h_{ij}=\mathrm{II}(e_{i},e_{j})$, $i,j\in\{\overline{1},2,\overline{2},\dots,\overline{n}\}$. Each of the monomials of $P_{k,q}$ containing only entries of the form $h_{ii}$ contains the factor $h_{\overline{1}\overline{1}}$ and exactly $n+q-k-1$ factors of the form $h_{jj}h_{\overline{j}\overline{j}}$, $i\in\{\overline{1},2,\overline{2},\dots,\overline{n}\}, j\in\{2,\dots,n\}$. 
\end{lemma}
\begin{proof}From (\ref{defRefMob}) we have $\theta_{0}=\sum_{i=2}^{n}\alpha_{1i}\wedge\beta_{1i}$, $\theta_{1}=\sum_{i=2}^n(\alpha_{i}\wedge\beta_{1i}-\beta_{i}\wedge\alpha_{1i})$, $\theta_{2}=\sum_{i=2}^n\alpha_{i}\wedge\beta_{i}$.

For convenience we write $\alpha_{\overline{i}}=\beta_{i}$ and $\alpha_{1\overline{i}}=\beta_{1i}$ for $i\in\{2,\dots,n\}$. Using $\alpha_{1j}=\sum_{i\in I}h_{ij}\alpha_{i}$ for $j\in I:=\{\overline{1},2,\overline{2},\dots,\overline{n}\}$ yields
\begin{multline*}\varphi^*(\gamma_{k,q})=\frac{c_{n,k,q}}{2}\left(\sum_{j\in I}h_{\overline{1}j}\alpha_{j}\right)\wedge\left(\sum_{i=2}^{n}\sum_{j,l\in I}h_{ij}h_{\overline{i}l}\alpha_{j}\alpha_{l}\right)^{n+q-k-1}\wedge\\\wedge\left(\sum_{i=2}^{n}\left(\sum_{j\in I}h_{\overline{i}j}\alpha_{i}\alpha_{j}-\sum_{l\in I}h_{il}\alpha_{\overline{i}}\alpha_{l}\right)\right)^{k-2q}\wedge\left(\sum_{i=2}^n\alpha_{i}\alpha_{\overline{i}}\right)^q.\end{multline*}
Thus, $\varphi^{*}(\gamma_{k,q})=P_{k,q}dx$ with $P_{k,q}$ a polynomial of degree $2n-k-1$ in $h_{ij}$.
The terms in the previous expression containing only entries of type $h_{ii}$ are 
\begin{multline*}\frac{c_{n,k,q}}{2}h_{\overline{1}\overline{1}}\alpha_{\overline{1}}\wedge\left(\sum_{i=2}^{n}h_{ii}h_{\overline{i}\overline{i}}\alpha_{i}\alpha_{\overline{i}}\right)^{n+q-k-1}\wedge\\\wedge\left(\sum_{i=2}^{n}h_{\overline{i}\overline{i}}\alpha_{i}\alpha_{\overline{i}}-h_{ii}\alpha_{\overline{i}}\alpha_{i}\right)^{k-2q}\wedge\left(\sum_{i=2}^n\alpha_{i}\alpha_{\overline{i}}\right)^q,\end{multline*}
and the result follows.
\end{proof}

\begin{theorem}\label{variaciorplans}
In $\C^n$, the variation of the valuation $\phi_{n-r}$ is given by
\begin{align}\label{standard1}\nonumber\delta\phi_{n-r}&=\vol(G_{n-1,r}^{\C})\omega_{2r+1}(r+1)\binom{n-1}{r}^{-1}\binom{n}{r}^{-1}\cdot\\&\quad\cdot\left(\sum_{q=\max\{0,n-2r-1\}}^{n-r-1}\binom{2n-2r-2q-1}{n-r-q}\frac{1}{4^{n-r-q-1}}B_{2n-2r-1,q}\right).\end{align}
Furthermore,  the following equation between valuations holds
\begin{align}\label{cn}\phi_{n-r}&=\frac{\vol(G_{n-1,r}^{\C})\omega_{2r}}{\binom{n-1}{r}\binom{n}{r}}\left(\sum_{q=\max\{0,n-2r\}}^{n-r}\frac{1}{4^{n-r-q}}\binom{2n-2r-2q}{n-r-q}\mu_{2n-2r,q}\right)\!\!.\end{align}
\end{theorem}
\begin{proof}
In order to simplify the following computations, we consider
\begin{equation}\label{canvip}\mu'_{k,q}:=c_{n,k,q}^{-1}\mu_{k,q}\quad  (\mbox{for }k\neq 2q),\quad\mu_{2q,q}':=2c_{n,2q,q}^{-1}\mu_{2q,q}\end{equation}
and \begin{equation} B'_{k,q}=c_{n,k,q}^{-1} B_{k,q},\quad\Gamma'_{k,q}=2c_{n,k,q}^{-1}\Gamma_{k,q}.\end{equation}

By \cite{fu.indiana}, the functional $\phi_{n-r}$ is a valuation on $\C^n$ with degree of homogeneity $2n-2r$. Thus, by Proposition \ref{baseMn} it can be expressed as a linear combination of $\{\mu'_{2n-2r,q}\}$; i.e.
\begin{equation}\label{igual1}\phi_{n-r}=\sum_{q=\max\{0,n-2r\}}^{n-r-1}C_{q}\mu'_{2n-2r,q}+D\mu'_{2n-2r,n-r}\end{equation} for certain constants $C_{q}, D$ which we wish to determine.
This will be done by comparing the variation of both sides of this equality. From here on we assume $2r<n$. The case $2r\geq n$ can be treated in the same way, or can be reduced to the previous case by means of the Fourier transform (cf. \cite{bernig.fu} Sections 2.1, 3.1 and 3.2).

By Proposition \ref{variacions}, 
\begin{equation}\label{suma11}
\delta\phi_{n-r}=\sum_{q=n-2r-1}^{n-r-1}c_{q}B'_{2n-2r-1,q}+\sum_{q=n-2r}^{n-r-1}d_{q}\Gamma'_{2n-2r-1,q}
\end{equation} where the coefficients $c_{q}$ and $d_{q}$ can be expressed in terms of a linear combination with known coefficients of the variables $C_{q}$ and $D$, that still remain unknown.

For a compact domain $\Omega$ with smooth boundary, and a smooth vector field $X$, Proposition \ref{variacioMesura} gives \begin{equation}\label{variacio1}\delta_{X}\phi_{n-r}(\Omega)\!=\!\int_{\partial\Omega}\langle X, \mathbf{n}\rangle\!\int_{G_{n-1,r}^{\C}}\!\!\!\det(\mbox{II}|_{V})dVdx.\end{equation}
From Lemma \ref{entries} when pulling-back the form $\gamma_{k,q}$ from $N(\Omega)$ to $\partial \Omega$, one gets a polynomial expression $P_{k,q}$ of degree $2n-k-1$ in the coefficients $h_{ij}$  of $\mbox{II}$ with $i,j\in\{\overline{1},2,\overline{2},\dots,n,\overline{n}\}$. Moreover, for each $q$ the monomials in $P_{k,q}$ containing only entries of the form $h_{ii}$ contain the factor $h_{\overline{1}\overline{1}}=\mbox{II}(J\mathbf{n},J\mathbf{n})$ and do not appear in any other $P_{k,q'}$ with $q'\neq q$. Therefore, every non-trivial linear combination of $\{P_{k,q}\}_q$ must contain the variable  $h_{\overline{11}}$. On the other hand, the integral $\int_{G_{n-1,r}^{\C}}\det(\mbox{II}|_{V})dV$ is a polynomial of the second fundamental form $\mbox{II}$ restricted to the distribution $\D=\mathrm{span}\{ \mathbf{n},J\mathbf{n}\}^{\bot}$, hence a polynomial not involving $h_{\overline{1}\overline{1}}$. Comparing the expressions of (\ref{suma11}) and (\ref{variacio1}), it follows that $d_{q}=0$ for all $q\in\{n-2r,\dots,n-r-1\}$.

As $c_{q}$ and $d_{q}$ depend on $C_{q}$ and $D$, we will obtain the value of $c_{q}$ once we know the value of $C_{q}$ and $D$. We will get their value from the equalities $\{d_{q}=0\}$. Note that this gives $r$ equations, since $q$ runs from $n-2r$ to $n-r-1$ in \eqref{suma11}. As for the unknowns, we need to find $r$ constants $C_q$ plus the constant $D$ in \eqref{igual1}.

We will get an extra equation by taking $\mbox{II}|_{\D}=\Id$ and equating (\ref{variacio1}) to (\ref{suma11}). Then, for any pair $(n,r)$ we have a compatible linear system since constants in (\ref{igual1}) exist. Next we find the solution, and we show it is unique. 
 
\medskip
Let us relate explicitly the coefficients $\{c_{q}\}$ and $\{d_{q}\}$ in \eqref{suma11} with $C_{q}$ and $D$ in \eqref{igual1}. To simplify the range of the subscripts, we denote $d_{n-r-a}$ with $a=1,\dots,r$ and $c_{n-r-a}$ with $a=1,\dots,r+1$.

\vspace{0.2cm}
\emph{Coefficient $d_{n-r-1}$.} From the variation of $\mu'_{k,q}$ in $\C^n$ (Proposition \ref{variacions}), the coefficient of $\Gamma'_{2n-2r-1,n-r-1}$ comes from the variation of $\mu'_{2n-2r,n-r-1}$ and $\mu'_{2n-2r,n-r}$.  
Then, \begin{align}\label{d1}d_{n-r-1}=4C_{n-r-1}-2r(n-r)D.\end{align}

\emph{Coefficient $d_{n-r-a}$, $a=2,\dots,r$.} The coefficient of $\Gamma'_{2n-2r-1,n-r-a}$ comes from the variation of $\mu'_{2n-2r,n-r-a}$ and $\mu'_{2n-2r,n-r-a+1}$. Then,
\begin{align}\label{da}d_{n-r-a}=4a^2C_{n-r-a}-(r-a+1)(n-r-a+1)C_{n-r-a+1}.
\end{align}

\emph{Coefficient $c_{n-r-1}$.} The coefficient of $ B'_{2n-2r-1,n-r-1}$ comes from the variation of $\mu'_{2n-2r,n-r-1}$ and $\mu'_{2n-2r,n-r}$. Then,
\begin{equation}\label{c1}c_{n-r-1}=2(2r+1)(n-r)D-4C_{n-r-1}.\end{equation}

\emph{Coefficient $c_{n-r-a}$, $a=2,\dots,r-2$.} The coefficient of $ B'_{2n-2r-1,n-r-a}$ comes from the variation of $\mu'_{2n-2r,n-r-a}$ and $\mu'_{2n-2r,n-r-a+1}$. Then,
\begin{align}\label{ca}c_{n-r-a}&=-4a(2a-1)C_{n-r-a}+(2r-2a+3)(n-r-a+1)C_{n-r-a+1}.\end{align}

\emph{Coefficient $c_{n-2r-1}$.} The coefficient of $B'_{2n-2r-1,n-2r-1}$ comes from the variation of $\mu'_{2n-2r,n-2r}$. Then,
\begin{align}\label{cr}c_{n-2r-1}=(n-2r)C_{n-2r}.\end{align}

Now, we solve the linear system given by $\{d_{n-r-a}=0\}$ where $a\in\{1,\dots,r\}$. From equations (\ref{d1}) and (\ref{da}) the system is given by: $$\left\{\begin{array}{rcl}r(n-r)D&=&2C_{n-r-1}\\4a^2C_{n-r-a}&=&(n-r-a+1)(r-a+1)C_{n-r-a+1}.\end{array}\right.$$
Thus, \begin{align}\label{cd}C_{n-r-a}&=\frac{(n-r-a+1)\cdot\dots\cdot(n-r)(r-a+1)\dots\cdot r}{2\cdot 4^{a-1}a^2(a-1)^{2}\cdot\dots\cdot 1^2}D\nonumber\\&=\frac{(n-r)!r!}{2^{2a-1}(n-r-a)!(r-a)!a!a!}D\nonumber\\&=\frac{D}{2^{2a-1}}\binom{n-r}{a}\binom{r}{a}.\end{align}

To obtain the value of $D$, we calculate $\int_{G_{n-1,r}^{\C}}\det(\mbox{II}|_{V})_{p}dV$ and $\beta_{2n-2r-1,n-r-a}$ in case $\mbox{II}|_{\mathcal D}(p)=\lambda Id$ for $\lambda>0$, which occurs when $\Omega$ is a metric ball. On the one hand, we have
$$\int_{G_{n-1,r}^{\C}}\det(\lambda\Id|_{V})_{p}dV=\lambda^{2r}\vol(G_{n-1,r}^{\C}).$$
On the other hand, if $\mbox{II}|_{\D}=\lambda\Id$, then the connection forms satisfy $\alpha_{1i}=\lambda\omega_{i}$ and $\beta_{1i}=\lambda\omega_{i}$. Thus, $\theta_{1}=2\lambda\theta_{2}$ and $\theta_{0}=\lambda^2\theta_{2}$ and we obtain 
\begin{align*}c_{n,2n-2r-1,n-r-a}^{-1}\beta_{2n-2r-1,n-r-a}(p)&=\lambda^{2r}(\beta\wedge\theta_{0}^{r-a+1}\wedge\theta_{1}^{2a-2}\wedge\theta_{2}^{n-r-a})(p)\\&=2^{2a-2}\lambda^{2r}(\beta\wedge\theta_{2}^{n-1})(p)=2^{2a-2}\lambda^{2r}(n-1)!.\end{align*}
So, the equation $$\vol(G_{n-1,r}^{\C})=\sum_{a=1}^{r+1}c_{n-r-a}2^{2a-2}(n-1)!$$ must be satisfied.

Substituting equations (\ref{c1}), (\ref{ca}) and (\ref{cr}) in the last equation gives {\small \begin{align*} &\frac{\vol(G_{n-1,r}^{\C})}{(n-1)!}=(2(2r+1)(n-r)D-4C_{n-r-1})\\&\quad+\sum_{a=2}^{r}2^{2a-2}((2r-2a+3)(n-r+a+1)C_{n-r-a+1}-4a(2a-1)C_{n-r-a})\\&\quad+2^{2r}(n-2r)C_{n-2r}\\&=2(2r+1)(n-r)D+4C_{n-r-1}((2r-1)(n-r-1)-1)\\&\quad+\sum_{a=2}^{r-1}(-2^{2a-2}4a(2a-1)+2^{2a}(2r-2a+1)(n-r-a))C_{n-r-a}\\&\quad+C_{n-2r}(2^{2r}(n-2r)-2^{2r-2}4r(2r-1))\\&=2(2r+1)(n-r)D+\sum_{a=1}^{r}2^{2a}C_{n-r-a}((2r-2a+1)(n-r-a)-a(2a-1))
\\&\stackrel{(\ref{cd})}{=}\!2D\sum_{a=0}^{r}{n-r\choose a}{r\choose a}\left({(2r-2a+1)(n-r-a)-a(2a-1)}\right)
\\&=D\frac{2\,n!}{r!(n-r-1)!}.\end{align*}}
The last equality will be proved in Lemma \ref{sumLem1}.
Thus,  \begin{align*}D&=\frac{\vol(G_{n-1,r}^{\C})}{2\,n!}\binom{n-1}{r}^{-1},
\\C_{n-r-a}&=\frac{\vol(G_{n-1,r}^{\C})}{4^{a}n!}\binom{n-1}{r}^{-1}\binom{n-r}{a}\binom{r}{a}\end{align*}
and, for $2r< n$, we have 
{\small \begin{align*}\phi_{n-r}&=\sum_{a=1}^{r}C_{n-r-a}B'_{2n-2r,n-r-a}+D\Gamma'_{2n-2r,n-r}\\&=\frac{\vol(G_{n-1,r}^{\C})}{2\,n!}\binom{n-1}{r}^{-1}\!\!\!\left(\sum_{a=1}^{r}\binom{n-r}{a}\!\!\binom{r}{a}2^{-2a+1}B'_{2n-2r,n-r-a}\!\!+\Gamma'_{2n-2r,n-r}\right)\end{align*}}
and
{\small \begin{align*}\delta&_{X}\phi_{n-r}(\Omega)=(2(2r+1)(n-r)D-4C_{n-r-1}){B}'_{2n-2r-1,n-r-1}\\&\!\!+\!\!\sum_{a=2}^{r}((2r-2a+3)(n-r+a+1)C_{n-r-a+1}\!-\!4a(2a-1)C_{n-r-a})B'_{2n-2r-1,n-r-a}\\&\!\!+\!(n-2r)C_{n-2r}B'_{2n-2r-1,n-2r-1}\\&=\frac{\vol(G_{n-1,r}^{\C})}{n!}\binom{n-1}{r}^{-1}\left(\sum_{a=1}^{r+1}\binom{n-r}{a}\binom{r+1}{a}\frac{a}{4^{a-1}}B'_{2n-2r-1,n-r-a}\right).\end{align*}}

Recalling (\ref{canvip}) and (\ref{mus}) gives the result. 
\end{proof}

The following lemma completes the proof of Theorem \ref{variaciorplans}.
\begin{lemma}\label{sumLem1}For $n, r\in\N$ with $2r\leq n$, the following equality holds
$$\sum_{a=0}^{r}{n-r\choose a}{r\choose a}({(2r-2a+1)(n-r-a)-a(2a-1)})=\frac{\,n!}{r!(n-r-1)!}.$$
\end{lemma}
\begin{proof}
Define the function $F(r,a)$ as $$F(r,a)={n-r\choose a}{r\choose a}({(2r-2a+1)(n-r-a)-a(2a-1)})$$
for $a\leq r \leq n-a$ and $0$ otherwise. We must compute  $$f(r)=\sum_{a=0}^{\infty}F(r,a).$$
The following relation is straightforward to check
\begin{equation}\label{sum1}-(n-r-1)F(r,a) + (r+1)F(r+1,a) = G(r,a+1) - G(r,a),\end{equation}
where \begin{multline*}
       G(r,a)=\binom{r}{a-1}\binom{n-r-1}{a-1}(-2a((2n+1)r-n^2+n+1)\\
-4r^3+(6n-10)r^2+(-2n^2+13n-7)r-3n^2+6n-1)
\end{multline*}
if $a\leq r\leq n-a$, and $G(r,a)=0$ otherwise. 
By summing equation \eqref{sum1} from $a=0$ to $\infty$, we get the following recurrence equation for $f(r)$
$$(n-r-1)f(r)=(r+1)f(r+1),$$and the result follows.
\end{proof} 
The previous proof follows the so-called Zeilberger's algorithm (\cite{zeilberger}). In particular, the key equality \eqref{sum1} was obtained by using the software EKHAD available with the book \cite{a=b}.

\begin{remark}\label{secondProof}An alternative way to prove equations \eqref{standard1}, and \eqref{cn} is as follows. Formula (\ref{cn}) can be obtained from equations (36), (37) in \cite{bernig.fu}, and Definition 3.6 there. Indeed, our valuation $\phi_{n-r}$ coincides up to normalization with the valuation $s^r$ of \cite{bernig.fu}. Again,  one needs to apply the Zeilberger algorithm in order to simplify some coefficients.
Then one gets equation \eqref{standard1} from \eqref{cn} essentially by the same computations as in the previous proof (equations \eqref{d1} to \eqref{cr}).
\end{remark}

\subsection{In Hermitian space forms}
\begin{corollary}\label{variacioepsilon}Let $\Omega\in\mathcal R(\Mn)$ be a compact domain with $C^{1,1}$ almost everywhere smooth boundary. Let $X$ be a smooth vector field over $\Mn$. Then,
\begin{align}\label{eqvariacioepsilon}\delta_{X}\phi_{n-r}(\Omega)=\frac{\vol(G_{n-1,r}^{\C})\omega_{2r+1}(r+1)}{\binom{n-1}{r}\binom{n}{r}}\int_{\partial\Omega}\langle X,\mathbf{n}\rangle \varphi^{*}(\omega)
\end{align}
where
$$\omega=\sum_{q=\max\{0,n-2r-1\}}^{n-r-1}\binom{2n-2r-2q-1}{n-r-q}\frac{1}{4^{n-r-q-1}}\beta_{2n-2r-1,q}.$$
\end{corollary}

\begin{proof}
Let us begin with the case $\epsilon=0$. Comparing equation (\ref{standard1}) and Proposition \ref{variacioMesura} shows that
$$\int_{\partial\Omega} \langle X,\mathbf{n}\rangle \left( \int_{G_{n-1,r}^{\C}}\det(\mbox{II}|_{V}) dV\right)dx$$
equals the right hand side of equation \eqref{eqvariacioepsilon}. By taking a field $X$ that vanishes outside an arbitrarily small neighborhood
of a fixed $x\in\partial \Omega$, we deduce the following equality between
forms
\begin{align}\label{promig}&\left(\int_{G_{n-1,r}^{\C}}\det(\mbox{II}|_{V})dV\right)dx=\frac{\omega_{2r+1}}{{{n-1}\choose r}{n \choose r}}\vol(G_{n-1,r}^{\C})(r+1)\varphi^{*}(\omega).
\end{align}
This equation extends obviously to $\Mn$ without change.
Then, using Proposition \ref{variacioMesura} gives the result.
\end{proof}

\begin{remark}\label{solucio}
 As a by-product we have the following consequence of \eqref{promig}. Let $\mbox{II}$ be a real symmetric bilinear form in $\C^n$. Consider
\[
 Q_r(\mbox{II}):=\int_{G_{n,r}^\C}\det(\mbox{II}|_V)dV.
\]
Let $\sigma:\C^n\rightarrow \C^n$ be the endomorphism associated to $\mbox{II}$ through the scalar product of $\C^n$ by $\langle\sigma(u),v\rangle=-\mbox{II}(u,v)$. Let $\overline\sigma:\C^n\rightarrow\C^n\times\C^n$ be the graph map $\overline\sigma(u)=(u,\sigma(u))$.\ Then
\[
 Q_r(\mbox{II})\ d\mbox{vol}=\sum_{q=\max\{0,n-2r\}}^{n-r}\!\!\alpha_{n,q,r}\overline{\sigma}^*(\theta_0^{r-q+1}\wedge\theta_1^{2q-2}\wedge\theta_2^{n+1-r-q})
\]
where $d\mbox{vol}$ is the canonical volume form in $\C^n$, and $\theta_0,\theta_1,\theta_2$ are $2$-forms in $\C^n\times\C^n$ defined by \eqref{defRefMob}, and
\[
\alpha_{n,q,r}= \binom{2n-2r-2q+1}{n+1-r-q}
\frac{c_{n+1,2n-2r+1,n+1-r-q}}{4^{n-r-q}}.
\]
\end{remark}

\begin{theorem}\label{mesuresCKn}In $\Mn$, the following equation between valuations holds
\begin{align}\notag\phi_{n-r}=\frac{\vol(G_{n-1,r}^{\C})}{\binom{n-1}{r}}&\left(\sum_{j=n-r}^{n}\epsilon^{j-(n-r)}\frac{\omega_{2n-2j}}{\binom{n}{j}}\cdot\huge((j+r-n+1)\mu_{2j,j}+\right.\\&\qquad+\left.\sum_{q=\max\{0,2j-n\}}^{j-1}\frac{1}{4^{j-q}}\binom{2j-2q}{j-q}\mu_{2j,q}\huge)\right).\label{crofton}
\end{align} 
\end{theorem}
\begin{proof}

We first focus on the right hand side of (\ref{crofton}), which we rewrite as
\[
\mathcal C_r:=\frac{\mathrm{vol}(G_{n-1,r}^{\mathbb C})}{n!} {n-1 \choose r}^{-1}\{\epsilon^r (r+1)n!\vol 
\]\[+\sum_{j=n-r}^{n-1}\epsilon^{j-n+r}\left(\frac{j-n+r+1}{2}\mu_{2j,j}'+\sum_{q=\max\{0,2j-n\}}^{j-1}\frac{1}{4^{j-q}}{n-j \choose j-q}{j \choose q} \mu_{2j,q}'\right)\}.
\]
By Proposition \ref{variacions}
{\small \begin{equation}\label{qual}
\delta_X\mathcal C_r=\frac{\mathrm{vol}(G_{n-1,r}^{\mathbb C})}{n!} {n-1 \choose r}^{-1}[\epsilon^r n (r+1) B_{2n-1,n-1}'
\end{equation}
\[
+\sum_{j=n-r}^{n-1}\epsilon^{j-n+r}\frac{j-n+r+1}{2}\{-2(n-j)j\Gamma_{2j-1,j-1}'+2\epsilon(n-j-1)(j+1)\Gamma_{2j+1,j}'
\]\[+4(n-j+\frac{1}{2})j B_{2j-1,j-1}'
+4\epsilon\left(\frac{j+1}{2}-(n-j)(2j+\frac{3}{2})\right) B_{2j+1,j}'+4\epsilon^2(n-j-1)(j+\frac{3}{2}) B_{2j+3,j+1}'\}]
\]

\[
+\sum_{j=n-r}^{n-1}\sum_{q=\max\{0,2j-n\}}^{j-1}\frac{\epsilon^{j-n+r}}{4^{j-q}}{n-j\choose j-q}{j \choose q}\{(2j-2q)^2\Gamma_{2j-1,q}'
\]
\[-(n+q-2j) q \Gamma_{2j-1,q-1}'+2(n+q-2j+\frac{1}{2})q B_{2j-1,q-1}'-2(2j-2q)(2j-2q-1) B_{2j-1,q}'
\]
\[
+2\epsilon(2j-2q)(2j-2q-1)  B_{2j+1,q+1}'-2\epsilon (n-2j+q)(q+\frac{1}{2}) B_{2j+1,q}'\}.\]}
Next we show that the previous expression is independent of $\epsilon$; i.e. all the terms containing $\epsilon$ cancel out. 
We concentrate first on the terms with $B'_{k,q}$. By putting together similar terms, the forth and fifth lines of (\ref{qual}) are
\begin{equation*}
\sum_{h=n-r+1}^{n-1}\epsilon^{h-n+r}2\{(h-n+r+1)(n-h+\frac{1}{2})h+(h-n+r)(\frac{h}{2}-(n-h+1)(2h-\frac{1}{2}))\end{equation*}
\begin{equation}\label{tal}+(h-n+r+1)(n-h+1)(h-\frac{1}{2})\} B_{2h-1,h-1}'\end{equation}\[-\epsilon^r \{(r+2)n-1\} B_{2n-1,n-1}'+(2r+1)(n-r) B_{2n-2r-1,n-r-1}'.\]

By putting together similar terms,  the double sum in (\ref{qual}) (forgetting for the moment the terms with $ \Gamma'_{k,q}$) becomes
{\footnotesize\[
\sum_{h=n-r}^{n-1}\sum_{a=\max\{-1,2h-n-1\}}^{h-2}\frac{\epsilon^{h-n+r}}{4^{h-a-1}} {n-h\choose h-a-1}{h \choose a+1}2(n+a-2h+\frac{3}{2})(a+1) B_{2h-1,a}'
\]
\[
-\sum_{h=n-r}^{n-1}\sum_{a=\max\{0,2h-n\}}^{h-1}\frac{\epsilon^{h-n+r}}{4^{h-a}} {n-h\choose h-a}{h \choose a}2(2h-2a)(2h-2a-1) B_{2h-1,a}'
\]\[
+\sum_{h=n-r+1}^{n}\sum_{a=\max\{1,2h-n-1\}}^{h-1}\frac{\epsilon^{h-n+r}}{4^{h-a}} {n-h+1\choose h-a}{h-1 \choose a-1}2(2h-2a)(2h-2a-1)  B_{2h-1,a}'
\]
\[-\sum_{h=n-r+1}^{n}\sum_{a=\max\{0,2h-n-2\}}^{h-2}\frac{\epsilon^{h-n+r}}{4^{h-a-1}} {n-h+1\choose h-a-1}{h-1 \choose a}2 (n-2h+a+2)(a+\frac{1}{2}) B_{2h-1,a}'.\]}
Note that  the terms with $a=-1$ or $a=2h-n-2$ vanish, if they occur. Then, one checks that all the terms in the above expression cancel out except those with $h=n-r,n$, and those with $a=h-1$. Clearly the terms corresponding to $h=n-r$ are independent of $\epsilon$. The terms with $h=n$ sum up ${\epsilon^{r}} (n-1)  B_{2n-1,n-1}'$,
and together with the similar term appearing in (\ref{tal}) cancel out the first term in (\ref{qual}). Finally, the terms with $a=h-1$ are cancelled with the sum in (\ref{tal}).

With a similar but shorter analysis one checks that the multiples of $\Gamma_{k,q}'$ cancel out completely. This shows that (\ref{qual}) is independent of $\epsilon$. Hence, $\delta\mathcal C_r$ coincides with the right hand side of (\ref{standard1}) since we know this happens for $\epsilon=0$.

Let now  $\Omega\in\mathcal R(\Mn)$ be a compact domain with $C^{1,1}$ and almost everywhere smooth boundary $\partial \Omega$.  By the last paragraph and Corollary \ref{variacioepsilon} we have $\delta_X\phi_{n-r}(\Omega)=\delta_X\mathcal C_r(\Omega)$ for every smooth vector field $X$. This implies that $\phi_{n-r}(F_t(\Omega))-\mathcal C_r(F_t(\Omega))$ does not depend on $t$, being $F_t$ the flow of $X$.  Let us take $X$ such that its flow $F_t$ converges to a point $p$  as $t$ tends to infinity. Assume further that $X$ coincides with the radial vector field $\sum_i x^i{\partial\over\partial x^i}$ on a local chart around $p$.

Then, $\mathcal C_r(F_t(\Omega))$ tends to 0 as $t$ goes to infinity. Indeed, let $\tilde F_t$ be the flow of contactomorphisms of $S(\Mn)$ that cover $F_t$. Then $N(F_t(\Omega))=\tilde F_t(N(\Omega))$, so
\begin{equation}\label{convergencia}
 \mu_{2j,q}(F_t(\Omega))=\int_{N(F_t(\Omega))}\beta_{2j,q}=\int_{N(\Omega)}\tilde F_t^*\beta_{2j,q}\to 0
\end{equation}
for $j\neq 0,q$. Indeed, note that $F_t^*\beta$ converges to $0$, and so does its multiple $\widetilde F_t^*\beta_{k,q}$. For $q=j\neq 0$ one can argue similarly.

On the other hand, \eqref{formulot} shows that $\phi_{n-r}(F_t(\Omega))$ has vanishing limit when $t$ goes to $\infty$. Therefore $\phi_{n-r}(F_t(\Omega))$ and $\mathcal C_{r}(F_t(\Omega))$ coincide for every $t$.

It remains only to prove $\phi_{n-r}(\Omega)=\mathcal C_r(\Omega)$ for $\Omega\in\mathcal P(\Mn)$. To this end, let us consider the parallel sets $\Omega_t$ at distance $t\geq 0$. For small $t>0$, we have $\Omega_t\in \mathcal R(\Mn)$. Hence, $\phi_{n-r}(\Omega_t)=\mathcal C_r(\Omega_t)$ for $t>0$. We need to show that both $\phi_{n-r}(\Omega_t)$ and $\mathcal C_r(\Omega_t)$ are continuous functions on $t\geq 0$.

The continuity of  $\mathcal C_r(\Omega_t)$ follows easily from the fact that $N(\Omega_t)$ is the image of $N(\Omega)$ under the geodesic flow of $S(\Mn)$. As for $\phi_{n-r}(\Omega_t)$, the continuity follows from equation \eqref{formulot}.
\end{proof}
\begin{remark}
The coefficients of $\mu_{k,q}$ in (\ref{crofton}) were found by solving a linear system of equations. These equations were obtained by imposing that the variations of both sides in (\ref{crofton}) coincide.
\end{remark}

\section{The Gauss-Bonnet theorem}\label{five}
\begin{theorem}\label{gb}In $\Mn$, 
\begin{equation}\label{gbequation}\omega_{2n}\chi=\sum_{c=0}^{n}\epsilon^c\frac{\omega_{2n-2c}}{\binom{n}{c}}\!\!\left(\!\sum_{q=\max\{0,2c-n\}}^{c-1}\!\!\frac{1}{4^{c-q}}\binom{2c-2q}{c-q}\mu_{2c,q}\!+\!(c+1)\mu_{2c,c}\!\right).\end{equation} 
\end{theorem}
\begin{proof} For $\epsilon=0$ equation (\ref{gbequation}) is the well known Gauss-Bonnet formula in $\mathbb C^n\equiv\mathbb R^{2n}$. For  general $\epsilon$ we proceed analogously to the proof of Theorem \ref{mesuresCKn}. In fact, the same computations of the previous proof show (in case $r=n$) that the right hand side of (\ref{gbequation}) has null variation. 

Hence, both sides are constant along flows. As in the previous proof, we take a  flow $F_t$ converging to a point $p$, and given by the radial vector field near $p$. In particular, $F_t(\Omega)$ converges to  $\{p\}$ for every $\Omega\in\mathcal P(\Mn)$ (or $\mathcal R(\Mn)$). As in \eqref{convergencia}, $\mu_{2c,q}(F_t(\Omega))$ goes to  $0$ for $c\neq 0$. The remaining term $\mu_{0,0}(F_t(\Omega))$ tends to $\chi(\Omega)$. Indeed, for $t$ big enough, $F_t(\Omega)$ is contained in a ball around $p$ of arbitrarily small radius. The metric inside this ball is close to euclidean. Hence, the total curvature $\mu_{0,0}(F_t(\Omega))$ converges to the euclidean total curvature, which equals $\chi(\Omega)$ by the Gauss-Bonnet theorem.
\end{proof}

\begin{remark}It is intriguing to notice that the coefficients inside the second sum in formula (\ref{gbequation}) coincide with those of the Maclaurin expansion of $1/\sqrt{1-x}$. We do not have any interpretation of this fact.
\end{remark}

\begin{remark}
For $n=2$ and $n=3$, the Gauss-Bonnet formula in $\Mn$ given in Theorem \ref{gb} was already proved in \cite{park}.
\end{remark}

\begin{theorem}In $\Mn$,
\begin{equation}\label{simple}\omega_{2n}\chi=\,\epsilon\phi_{1}+\sum_{k=0}^{n}\epsilon^k \omega_{2n-2k}\binom{n}{k}^{-1}\mu_{2k,k}.\end{equation}
\end{theorem}
\begin{proof}

From Theorems \ref{mesuresCKn} and \ref{gb}, it follows that 
\small{\begin{align*}\chi&=\sum_{c=0}^{n-1}\frac{\epsilon^c\,c!}{\pi^c}\left(\sum_{q=\max\{0,2c-n\}}^{c-1}\frac{1}{4^{c-q}}\binom{2c-2q}{c-q}\mu_{2c,q}+(c+1)\mu_{2c,c}\right)
\\&=\frac{\epsilon\,n!}{\pi^n}\sum_{c=1}^{n}\frac{\epsilon^{c-1}\,c!\pi^{n-c}}{n!}\!\!\left(\sum_{q=\max\{0,2c-n\}}^{c-1}\!\!\!\!\frac{1}{4^{c-q}}\binom{2c-2q}{c-q}\mu_{2c,q}+c\mu_{2c,c}\right)\\&\quad+\frac{\epsilon\,n!}{\pi^n}\sum_{c=1}^{n}\frac{\epsilon^{c-1}c!\pi^{n-c}}{n!}\mu_{2c,c}
\\&=\frac{\epsilon\,n!}{\pi^n}\phi_{1}+\sum_{c=1}^{n}\frac{\epsilon^c\,c!}{\pi^c}\mu_{2c,c}.
\end{align*}} 
\end{proof}

\begin{remark}
Recall that the first term in the sum of \eqref{simple} corresponds to the Gauss curvature integral $\mu_{0}$.
\end{remark}

\begin{remark}Let $\Omega\subset\mathcal{P}(\C^n)$. From Theorem \ref{variaciorplans} we get an expression for $$\int_{\L_{r}^{\C}}\mu_{0,0}(\Omega\cap L_{r})dL_{r}$$ in terms of the hermitian intrinsic volumes. This shows that $\mu_{0,0}$ does not have the so-called reproductive property. A similar fact happens with $\mu_{2n-2,n-1}$ and $\mu_{2n-2,n-2}$  (cf. \cite{abardia.m1}).
\end{remark}

\def\cprime{$'$}

\end{document}